\documentclass{aims}
\usepackage{amsmath}
  \usepackage{paralist}
  \usepackage{graphics} %% add this and next lines if pictures should be in esp format
  \usepackage{epsfig} %For pictures: screened artwork should be set up with an 85 or 100 line screen
\usepackage{graphicx}  \usepackage{epstopdf}%This is to transfer .eps figure to .pdf figure; please compile your paper using PDFLeTex or PDFTeXify.
 \usepackage[colorlinks=true]{hyperref}
\hypersetup{urlcolor=blue, citecolor=red}

\usepackage{mathtools}
\usepackage{wasysym}

\def\currentvolume{X}
 \def\currentissue{X}
  \def\currentyear{200X}
   \def\currentmonth{XX}
    \def\ppages{X--XX}
     \def\DOI{10.3934/xx.xx.xx.xx}

\newtheorem{theorem}{Theorem}[section]

\newtheorem{lemma}[theorem]{Lemma}

\theoremstyle{definition}

\newtheorem{remark}{Remark}

\newcommand\R{\mathbb{R}}
\newcommand\Z{\mathbb{Z}}

\newcommand\C{\mathbb{C}}

\title[Universality of Euler equation] %Use the shortened version of the full title
      {On the universality of the incompressible Euler equation on compact manifolds}

\author[Terence Tao]{}


 \email{tao@math.ucla.edu}

\begin{document}
\maketitle

% Enter the first author's name and address:
\centerline{\scshape Terence Tao}
\medskip
{\footnotesize
% please put the address of the first author
 \centerline{Dept of Mathematics UCLA}
   \centerline{405 Hilgard Ave}
   \centerline{Los Angeles, CA 90095}
} % Do not forget to end the {\footnotesize by the sign }

\bigskip

% The name of the associate editor will be entered by an editorial staff
% "Communicated by the associate editor name" is not needed for special issue.
 \centerline{(Communicated by the associate editor name)}

%The abstract of your paper
\begin{abstract}
The incompressible Euler equations on a compact Riemannian manifold $(M,g)$ take the form
\begin{align*}
\partial_t u + \nabla_u u &= - \mathrm{grad}_g p \\
\mathrm{div}_g u &= 0.
\end{align*}
We show that any quadratic ODE $\partial_t y = B(y,y)$, where $B \colon \R^n \times \R^n \to \R^n$ is a symmetric bilinear map, can be linearly embedded into the incompressible Euler equations for some manifold $M$ if and only if $B$ obeys the cancellation 
condition $\langle B(y,y), y \rangle = 0$ for some positive definite inner product $\langle,\rangle$ on $\R^n$.  This allows one to construct explicit solutions to the Euler equations with various dynamical features, such as quasiperiodic solutions, or solutions that transition from one steady state to another, and provides evidence for the ``Turing universality'' of such Euler flows.
\end{abstract}

%The title of your section 1
\section{Introduction}

Let $(M,g)$ be a compact connected smooth orientable Riemannian manifold without boundary (which we henceforth abbreviate as \emph{compact Riemannian manifold}).  The incompressible Euler equations on $M$ take the form
\begin{equation}\label{euler}
\begin{split}
\partial_t u + \nabla_u u &= - \mathrm{grad}_g p \\
\mathrm{div}_g u &= 0
\end{split}
\end{equation}
where for each time $t$, $u(t) \in \Gamma(TM)$ is a smooth vector field on $M$ (the \emph{velocity field}), $p \in C^\infty(M)$ is a scalar field (the \emph{pressure}), $\mathrm{grad}_g$ is the gradient with respect to the metric $g$, $\mathrm{div}_g$ is the divergence with respect to $g$ (or the volume form associated with $g$), and $\nabla$ is the Levi-Civita connection.  These equations may be interpreted as geodesic flow on the infinite-dimensional manifold of volume-preserving diffeomorphisms of $M$; see \cite{ebin}.  The analysis in \cite{ebin} also covers the case when $M$ is non-orientable, non-compact, or contains a boundary; but for this paper we will restrict attention to the case of compact Riemannian manifolds for simplicity.  We will also only consider classical (i.e., smooth) solutions to \eqref{euler} in this paper, in particular there will be no discussion of weak solutions.

Formally, one can use Hodge theory to eliminate the modified pressure term from the Euler equations, and rewrite \eqref{euler} as 
\begin{equation}\label{euler-2} 
\partial_t u = B_E(u,u)
\end{equation}
where the symmetric bilinear operator $B_E(u_1, u_2)$ is defined as the orthogonal projection of $- \frac{1}{2} ( \nabla_{u_1} u_2 + \nabla_{u_2} u_1 )$ onto divergence-free vector fields.  A similar use of Hodge theory can be used to solve for $p$ (up to constants) as a quadratic function of $u$.

In \cite{tao-navier}, the author investigated modified Euler equations of the form 
\begin{equation}\label{mode}
 \partial_t u = \tilde B_E(u,u)
\end{equation}
where $\tilde B_E$ was an ``averaged'' version of $B_E$, as well as the analogous modification of the Navier-Stokes equations.  By carefully choosing the averaged operator $\tilde B_E$, one was able to embed inside \eqref{mode} some carefully chosen system of ordinary differential equations (ODE) of the form\footnote{See also \cite{brom}, \cite{dick}, \cite{kap} for some prior literature on the evolution of equations of this type.  We thank Joseph Malkoun for these references, and for sharing some unpublished work on equations of the form \eqref{yby} that obey variants of the condition \eqref{yo}.}
\begin{equation}\label{yby}
 \partial_t y = B(y,y) 
\end{equation}
where $y$ now takes values in a finite-dimensional vector space $\R^n$, and $B \colon \R^n \times \R^n \to \R^n$ is a bilinear map, which we can take without loss of generality to be symmetric (otherwise we can just replace $B$ with its symmetrisation $(y,z) \mapsto \frac{B(y,z) + B(z,y)}{2}$), and which obeyed a conservation law
\begin{equation}\label{yo}
 \langle B(y,y), y \rangle = 0
\end{equation}
for all $y \in \R^n$ and some positive definite inner product $\langle,\rangle \colon \R^n \times \R^n \to \R$, which in particular gives rise to the energy conservation law
$$ \partial_t \langle y, y \rangle = 0.$$
In particular, as the level sets of $\langle y, y\rangle$ are compact, this implies that solutions to \eqref{yby} exist globally in time.
On the other hand, an infinite-dimensional version of \eqref{yby} was constructed in \cite{tao-navier} that exhibited finite time blowup even in the presence of some dissipation, which was then used to establish finite time blowup for an averaged version of the Navier-Stokes equations.

Based on this, one may ask the question of whether ODE of the form \eqref{yby} may be embedded into the true Euler equations \eqref{euler}, (or \eqref{euler-2}), as opposed to the artificially modified Euler equations \eqref{mode}.  More formally, we say that a bilinear symmetric map $B \colon \R^n \times \R^n \to \R^n$ can be \emph{embedded into the Euler equations} for a compact Riemannian manifold $(M,g)$ if there exists an injective linear map $U \colon \R^n \to \Gamma(TM)$, and a map $P \colon \R^n \to C^\infty(M)$ into the spaces $\Gamma(TM)$, $C^\infty(M)$ of smooth vector fields and smooth scalar fields on $M$ respectively, such that whenever $t \mapsto y(t)$ is a solution to the ODE \eqref{yby} on some time interval $I$, the fields $u \colon t \mapsto U(y(t))$ and $p \colon t \mapsto P(y(t))$ solve the Euler equations \eqref{euler} on the same time interval $I$.  Equivalently (by the Picard existence theorem), we have the equations
\begin{equation}\label{euler-embed}
\begin{split}
U(B(y,y)) + \nabla_{U(y)} U(y) &= - \mathrm{grad}_g P(y) \\
\mathrm{div}_g U(y) &= 0
\end{split}
\end{equation}
for all $y \in \R^n$; the image $U(\R^n)$ of $U$ is then a finite-dimensional invariant subspace of the state space for the flow \eqref{euler-2}.  Informally, if $B$ can be embedded into the Euler equations, then we may accurately simulate the dynamics of the system \eqref{yby} by using an incompressible fluid on a suitable compact Riemannian manifold.

The main result of this paper is to give a complete answer to this question, if one is given the freedom to choose the manifold $M$:

\begin{theorem}[Embeddability criterion]\label{main}  Let $B \colon \R^n \times \R^n \to \R^n$ be a symmetric bilinear map.  Then the following are equivalent:
\begin{itemize}
\item[(i)]  $B$ can be embedded into the Euler equations for some Riemannian manifold $(M,g)$.
\item[(ii)]  There exists a positive definite inner product $\langle, \rangle \colon \R^n \times \R^n \to \R$ on $\R^n$ such that one has the cancellation \eqref{yo} for all $y \in \R^n$.  
\end{itemize}
\end{theorem}

Thus, for instance, one cannot embed $B$ into the Euler equations for any manifold $M$ if there is a non-zero $y \in \R^n$ such that $B(y,y)$ is a non-zero scalar multiple of $y$.  Informally, Theorem \ref{main} indicates that energy conservation, as well as the bilinear nature of the nonlinearity, are the \emph{only} constraint on the dynamics of the Euler equations, at least insofar as the dynamics of finite-dimensional invariant subspaces for these equations are concerned.

The derivation from (ii) to (i) is an easy consequence of the energy conservation for the Euler equations.  Indeed, for any smooth solution $(u,p)$ to the Euler equations \eqref{euler}, a standard integration by parts reveals the energy conservation law
$$ \partial_t \langle u(t), u(t) \rangle_{L^2(M)} = 0,$$
where
$$ \langle u, v \rangle_{L^2(M)} \coloneqq \int_M g(x)( u(x), v(x) )\ d\mathrm{vol}_g(x) $$
is the usual inner product between (smooth) vector fields on $M$ using the volume form $d\mathrm{vol}_g$ associated to the Riemannian metric $g$ on the orientable manifold $M$.  In particular, if $B \colon \R^n \times \R^n \to \R^n$ is embedded into the Euler equations via smooth linear maps $U \colon \R^n \to \Gamma(TM)$, $P \colon \R^n \to C^\infty(M)$, then we have
$$ \partial_t \langle y(t), y(t) \rangle_{U} = 0$$
where $\langle, \rangle_U \colon \R^n \times \R^n \to \R$ is the bilinear form
$$ \langle y, z \rangle_U \coloneqq \langle U(y), U(z) \rangle_{L^2(M)}.$$
As $U$ is injective, $\langle, \rangle_U$ is positive definite.  From \eqref{yby} and the chain rule (and the Picard existence theorem), we conclude the cancellation law \eqref{yo} (with $\langle, \rangle$ being set to $\langle, \rangle_U$).

The derivation of (i) from (ii) is more difficult, and will occupy the bulk of this paper.  The proof proceeds by a series of reductions.  Firstly, one reduces the reliance on the metric field $g$ by rephrasing some components of the Euler equations in terms of the \emph{covelocity} $V$ instead of the velocity $U$; this is related to the well-known \emph{vorticity formulation} of the Euler equations, although the vorticity (the exterior derivative of the covelocity) will play a relatively minor role in our analysis.  Once one does this, it becomes relatively easy to eliminate the role of the metric $g$ (especially given the freedom to modify the manifold $M$), as long as a certain ``Gram bilinear form'' relating $U$ and $V$ is positive definite.  By selecting a certain ansatz for $U$ and $V$ in terms of a system of scalar fields $F$, one can reduce the Euler equations to a system of transport equations.  This system is in general overdetermined, which makes it unlikely that one can find large numbers of solutions to the system; however, it turns out that there is a highly symmetric explicit solution, based on the Lie group structure of the orthogonal group $SO(n)$, that avoids the overdeterminacy.  (As a consequence, the manifold $M$ used to model the equation \eqref{yby} can be explicitly given as $SO(n) \times (\R/\Z)^{n+1}$, although the metric $g$ one places on $M$ is somewhat artificial.)

The implication of (i) from (ii) allows one to construct quite explicit solutions to the Euler equations with some interesting dynamical behavior.  For instance, for any $\alpha \in \R$, the system of ODE
\begin{equation}\label{moe}
\begin{split}
\partial_t y_1 &= \alpha y_2 y_3 \\
\partial_t y_2 &= - \alpha y_1 y_3 \\
\partial_t y_3 &= 0
\end{split}
\end{equation}
(which is referred to as the ``rotor gate'' in \cite{tao-navier}) is of the form \eqref{yby} with a bilinear form obeying \eqref{yo} (with the Euclidean inner product on $\R^3$), and admits the explicit periodic solutions
\begin{align*}
y_1(t) &= A \sin(\alpha \omega t + \theta) \\
y_2(t) &= A \cos(\alpha \omega t + \theta) \\
y_3(t) &= \omega
\end{align*}
for any $A, \omega, \theta \in \R$.  Applying Theorem \ref{main}, we conclude the existence of a compact Riemannian manifold $(M,g)$ (in fact one can take $M = SO(3) \times (\R/\Z)^4$) and (explicitly constructible) linearly independent divergence-free vector fields $u_1, u_2, u_3 \in \Gamma(TM)$, such that for any $A,\omega,\theta \in \R$, one has periodic solutions to the Euler equations \eqref{euler} on $(M,g)$ with velocity field
$$ u(t) = A \sin(\alpha \omega t + \theta) u_1 + A \cos(\alpha \omega t + \theta) u_2 + \omega u_3,$$
as well as a pressure field $p(t)$ that can be explicitly computed, though we will not do so here.  By taking tensor powers of \eqref{moe}, one can similarly construct a compact Riemannian manifold that admits quasiperiodic solutions; we leave the details to the interested reader.

In a similar vein, for any $\alpha \in \R$, the system of ODE
\begin{align*}
\partial_t y_1 &= -\alpha y_1 y_2 \\
\partial_t y_2 &= \alpha y_1^2 
\end{align*}
(referred to as the ``pump gate'' in \cite{tao-navier}) also is of the form \eqref{yby} obeying \eqref{yo}, and has the explicit solutions
\begin{align*}
y_1(t) &= A \operatorname{sech}(A \alpha t)\\ 
y_2(t) &= A \operatorname{tanh}(A \alpha t) 
\end{align*}
for any $A \in \R$, which converges to the steady state $(y_1,y_2) = (0,+A)$ as $t \to +\infty$ and $(y_1,y_2) = (0,-A)$ as $t \to -\infty$.  Applying Theorem \ref{main}, we conclude the existence of a compact Riemannian manifold $(M,g)$ (in fact one can take $M = SO(2) \times (\R/\Z)^3 \equiv (\R/\Z)^4$) and linearly independent divergence-free vector fields $u_1,u_2$, such that for any $A \in \R$, one has solutions to \eqref{euler} on $(M,g)$ with velocity field
$$ u(t) = A \operatorname{sech}(A \alpha t) u_1 + A \operatorname{tanh}(A \alpha t) u_2$$
which approach the steady state $u_2$ as $t \to +\infty$ and $u_1$ as $t \to -\infty$.  Similarly for the ``amplifier gate''
\begin{align*}
\partial_t y_1 &= -\alpha y_2^2 \\
\partial_t y_2 &= \alpha y_1 y_2
\end{align*}
that is also discussed in \cite{tao-navier}.

By coupling together a finite number of such gates, one can create (finite-dimensional fragments of) inviscid \emph{shell models}; see e.g. \cite{shell1}, \cite{shell2}, \cite{shell3}, \cite{shell4}, \cite{ds}, \cite{kp}, \cite{fp} for some examples of such models.  Theorem \ref{main} then allows us to embed any such finite-dimensional fragment of a shell model into the Euler flow of a compact manifold, although as before the dimension of the manifold will depend on the dimension of the fragment.  Informally, this provides some heuristic support for the use of such systems as simplified toy models for the study of Euler equations.

In \cite{tao-navier}, a carefully chosen coupling of such gates was used to create a system that performed a delayed, but abrupt, transfer of energy from one mode to another; again, one can use Theorem \ref{main} to then recreate the same dynamics in the Euler equations on some compact manifold.  Unfortunately, the infinite-dimensional ODE used to create finite time blowup in \cite{tao-navier} lies outside of the range of applicability of Theorem \ref{main} (note in particular that the dimension of the spatial manifold $M$ constructed in Theorem \ref{main} will depend on the dimension $n$ of the ODE).  Nevertheless, it raises the distinct possibility that one can somehow adapt the methods in \cite{tao-navier} to demonstrate finite-time blowup for the true Euler equations (as opposed to an artificially averaged Euler equation) on some finite (but high) dimensional Riemannian manifold $(M,g)$.  One possible step in this direction would be to construct a quadratic ODE \eqref{yby} (obeying \eqref{yo}) which exhibited ``Turing universality'' in the spirit of \cite[Corollary 1.11]{tao-univ}.  This appears to be somewhat challenging, due to the fact that the ODE \eqref{yby} behaves like an ``analog'' computer rather than a ``digital'' one; on the other hand, a primitive example of an ``analog-to-digital converter'' in this setting was used in \cite{tao-navier}, so the author does not view the possibility of constructing a Turing universal quadratic ODE to be totally out of the question.

Somewhat amusingly, Theorem \ref{main} also allows one to embed the Euler equations 
\begin{align*}
I_1 \partial_t \omega_1 + (I_3 - I_2) \omega_2 \omega_3 &= 0 \\
I_2 \partial_t \omega_2 + (I_1 - I_3) \omega_3 \omega_1 &= 0 \\
I_3 \partial_t \omega_3 + (I_2 - I_1) \omega_1 \omega_2 &= 0 
\end{align*}
for the free motion of a three-dimensional rigid body with moments of inertia $I_1, I_2, I_3 > 0$, into the Euler equations for incompressible fluid flow on some Riemannian manifold $(M,g)$; the inner product in this case is associated to the total kinetic energy $\frac{1}{2} I_1 \omega_1^2 + \frac{1}{2} I_2 \omega_2^2 + \frac{1}{3} I_3^2$.  The two Euler equations were previously observed to be analogous in \cite{arnold}, as both could (formally, at least) be viewed as geodesic flow on a Lie group.  In particular, instability effects such as those arising from the ``tennis racket theorem'' \cite{tennis} may be seen in the equations of incompressible fluid flow on $(M,g)$.

We stress that the Riemannian manifold $(M,g)$ produced by this theorem will depend on the choice of $B$ (and on the dimension $n$).  In particular, our methods are unable to embed arbitrary ODE of the form \eqref{yby} into the Euler equations on a flat manifold such as a torus, though it would be interesting to know if this was indeed possible.  Furthermore the manifolds $M$ used are rather high dimensional (the dimension grows quadratically in $n$); we unfortunately have nothing to say about the dynamics of three-dimensional Euler equations (where there are potentially more constraints on the dynamics, for instance due to helicity conservation in the case of flat domains).

The author is supported by NSF grant DMS-1266164 and by a Simons Investigator Award.  The author also thanks the commenters on his blog for some corrections.

\section{First reduction: covelocity formulation}

We begin the proof of Theorem \ref{main}. The derivation of (ii) from (i) was already established in the introduction, so we focus on the derivation of (i) from (ii).  We begin with an easy reduction: by a linear change of variable (using an orthonormal basis associated to the positive definite inner product $\langle, \rangle$), we may assume without loss of generality that $\langle, \rangle$ is the Euclidean inner product $\langle, \rangle_{\R^n}$ on $\R^n$.  It will now suffice to find a compact Riemannian manifold $(M,g)$, an injective linear map $U \colon \R^n \to \Gamma(TM)$ to the space of vector fields on $M$, and a symmetric bilinear map $P \colon \R^n \times \R^m \to C^\infty(M)$ to the scalar fields of $M$, which solve the system of equations
\begin{equation}\label{euler-coef}
\begin{split}
U(B(y,y)) + \nabla_{U(y)} U(y) &= - \mathrm{grad}_g P(y,y) \\
\mathrm{div}_g U(y) &= 0
\end{split}
\end{equation}
on $M$ for all $y \in \R^n$.

The next reduction involves the introduction of the \emph{covelocity field} $V \colon \R^n \to \Gamma(T^* M)$, defined as the dual $1$-forms to the vector fields $U$ with respect to the metric $g$, thus
$$ g( U(y), X ) = V(y)( X )$$
for any vector field $X \in \Gamma(TM)$ and $y \in \R^n$.  We abbreviate this as
$$ V(y) \coloneqq g \cdot U(y);$$
in Penrose abstract index notation (using Greek indices $\alpha,\beta,\gamma$ for the abstract coordinates on $M$), this would be
$$ V(y)_\alpha = g_{\alpha \beta} U(y)^\beta.$$
Using $g^{\alpha \beta}$ to denote the inverse $g^{-1}$ of the metric $g = g_{\alpha \beta}$, we then have
$$ U(y)^\beta = g^{\alpha \beta} V(y)_{\alpha}$$
which we abbreviate as
$$ U(y) = g^{-1} \cdot V(y).$$
Clearly, the map $U$ will be injective if and only if $V$ is.

It is also convenient to introduce the \emph{vorticity field} $dV \colon \R^n \to  \Gamma(\bigwedge^2 T^* M)$, which are the $2$-forms generated by applying an exterior derivative to the covelocity fields $V$.
In Penrose abstract index notation, this is
$$ dV(y)_{\alpha \beta} = \partial_\alpha V(y)_\beta - \partial_\beta V(y)_\alpha.$$
Using the Levi-Civita connection $\nabla$, we can also write
\begin{equation}\label{oma}
 dV(y)_{\alpha \beta} = \nabla_\alpha V(y)_\beta - \nabla_\beta V(y)_\alpha.
\end{equation}
Now we consider the $1$-forms
$$ U(y) \invneg dV(y) $$
for $y \in \R^n$, formed by contracting the $2$-form $dV(y)$ by the vector field $U(y)$. In Penrose abstract index notation, we have
$$ (U(y) \invneg dV(y))_\beta = U(y)^\alpha dV(y)_{\alpha \beta},$$
which by \eqref{oma} is equal to 
$$ U(y)^\alpha \nabla_\alpha V(y)_\beta - U(y)^\alpha \nabla_\beta V(y)_\alpha.$$
The first term can be rewritten as $(\nabla_{U(y)} V(y))_\beta$.  Since the metric $g$ is parallel to the Levi-Civita connection, and $U(y)^\alpha V(y)_\alpha = g(U(y),U(y))$, we also see from the product rule that
$$ U(y)^\alpha \nabla_\beta V(y)_\alpha =\frac{1}{2} \partial_\beta g(U(y), U(y))$$
and so we conclude the $1$-form identity
\begin{equation}\label{dgub}
 U(y) \invneg dV(y) = \nabla_{U(y)} V(y) -\frac{1}{2} d g(U(y), U(y))
\end{equation}
for all $y \in \R^n$.  On the other hand, by applying the metric $g$ to the first equation of \eqref{euler-coef} to convert vectors to $1$-forms, and recalling that the metric $g$ is parallel to the Levi-Civita connection, we see that this equation is equivalent to
$$
V(B(y,y)) + \nabla_{U(y)} V(y) = - d P(y,y).$$
Using \eqref{dgub}, we can rewrite this as
$$
V(B(y,y)) + U(y) \invneg dV(y) = - d P'(y,y)$$
where $P' \colon \R^n \times \R^n \to C^\infty(M)$ is the modified pressure field
$$P'(y,y') \coloneqq P(y,y') + \frac{1}{2} g(U(y),U(y')).$$
Clearly one can reconstruct the pressure field $P$ from the modified pressure $P'$ and from $g, U$ by the formula
$$ P(y,y') = P'(y,y') - \frac{1}{2} g(U(y),U(y')).$$
We have thus reduced Theorem \ref{main} to the following statement.

\begin{theorem}[First reduction]\label{first}  Let $B \colon \R^n \times \R^n \to \R^n$ be a symmetric bilinear map obeying \eqref{yo}.  Then there exists a compact Riemannian manifold $(M,g)$, an injective linear map $V \colon \R^n \to \Gamma(T^* M)$, a linear map $U \colon \R^n \to \Gamma(TM)$, and a symmetric bilinear map $P' \colon \R^n \times \R^n \to \C^\infty(M)$ obeying the equations
\begin{align}
V(B(y,y)) + U(y) \invneg dV(y) &= - d P'(y,y) \label{f1}\\
V(y) &= g \cdot U(y) \label{f2}\\
\mathrm{div}_g U(y) &= 0 \label{f4}
\end{align}
on $M$ for all $y \in \R^n$.
\end{theorem}

It remains to establish Theorem \ref{first}.  This will be the objective of the next four sections of the paper.

\section{Second reduction: decoupling the metric and volume form}

An inspection of the system \eqref{f1}-\eqref{f4} that one is trying to solve in Theorem \ref{first} reveals that the metric $g$ is only appearing in two places: in the equation \eqref{f2} linking the velocity field $U$ with the covelocity field $V$, and in the divergence-free condition \eqref{f4}.  However, the influence of the metric $g$ on \eqref{f4} is quite mild, since the divergence operator $\mathrm{div}_g$ only depends on $g$ through the volume form $d\mathrm{vol}_g$, as can be seen by the integration by parts identity
$$ \int_M f \mathrm{div}_g(X)\ d\mathrm{vol}_g = - \int_M df(X)\ d\mathrm{vol}_g,$$
valid for any vector field $X \in \Gamma(TM)$ and scalar field $f \in C^\infty(M)$.  Indeed, one can similarly define the divergence operator $\mathrm{div}_{\mathrm{vol}^m}$ with respect to any everywhere positive volume form $\mathrm{vol}^m \in \Gamma(\bigwedge^m T^* M)$ on $M$ (with $m$ denoting the dimension of $M$, and using the orientation of $M$ to define positivity) by the same formula:
$$ \int_M f \mathrm{div}_{\mathrm{vol}^m}(X)\ \mathrm{vol}^m = - \int_M df(X)\ \mathrm{vol}^m.$$
Equivalently, the volume form $\mathrm{vol}^m$ induces a Hodge duality relationship between $k$-vector fields and $m-k$-forms for any $0 \leq k \leq m$, and the divergence operator $\mathrm{div}_{\mathrm{vol}^m}$ is the conjugate of the exterior derivative $d$ by this Hodge dual operation.

Define the \emph{determinant} $\mathrm{det}_{\mathrm{vol}^m}(g) \in C^\infty(M)$ of a Riemannian metric $g$ with respect to a everywhere positive volume form $\mathrm{vol}^m$ to be the unique positive smooth scalar function such that
$$ d\mathrm{vol}_g= \mathrm{det}_{\mathrm{vol}^m}(g)^{1/2} \mathrm{vol}^m.$$
For instance, if $\mathrm{vol}^m$ is the Euclidean volume form on $\R^m$ and $g_{ij} = g(e_i,e_j)$ are the standard coefficients of the metric $g$, then $\mathrm{det}_{\mathrm{vol}^m}(g)$ is just the usual determinant of the $m \times m$ matrix $(g_{ij})_{1 \leq i,j \leq m}$.  One can then split the equation \eqref{f4} into two equations
\begin{align*}
\mathrm{div}_{\mathrm{vol}^m} u_a &= 0 \\
\mathrm{det}_{\mathrm{vol}^m} g &= 1
\end{align*}
involving an auxiliary volume form $\mathrm{vol}^m$.

In this section, we exploit the freedom to increase the dimension of $M$ to eliminate the determinant condition $\mathrm{det}_{\mathrm{vol}^m} g = 1$, thus decoupling the metric from the volume form.  More precisely, we deduce Theorem \ref{first} from

\begin{theorem}[Second reduction]\label{second}  Let $B \colon \R^n \times \R^n \to \R^n$ be a symmetric bilinear map obeying \eqref{yo}.  Then there exists a compact Riemannian manifold $(M,g)$ of some dimension $m$, an injective linear map $V \colon \R^n \to \Gamma(T^* M)$, a linear map $U \colon \R^n \to \Gamma(TM)$, a symmetric bilinear map $P' \colon \R^n \times \R^n \to C^\infty(M)$, and an everywhere positive volume form $\mathrm{vol}^m \in \Gamma(\bigwedge^m T^* M)$ obeying the equations
\begin{align}
V(B(y,y)) + U(y) \invneg dV(y) &= -dP'(y,y) \label{g1}\\
V(y) &= g \cdot U(y) \label{g2}\\
\mathrm{div}_{\mathrm{vol}^m} U(y) &= 0 \label{g4}
\end{align}
on $M$ for all $y \in \R^n$.
\end{theorem}

Let us now see how Theorem \ref{second} implies Theorem \ref{first}.  Let $n, B$ obey the hypotheses of Theorem \ref{first}, and let $M, g, m, V, U, P', \mathrm{vol}^m$ be the objects associated to $n, B$ by Theorem \ref{second}.  We then define the $m+1$-dimensional Riemannian manifold $(\tilde M, \tilde g)$ by setting 
$$ \tilde M \coloneqq M \times (\R/\Z)$$
with metric
\begin{equation}\label{gxy}
 \tilde g( (X,u), (Y,v) ) \coloneqq g(X, Y) + (\mathrm{det}_{\mathrm{vol}} g)^{-1} uv 
\end{equation}
at any point $(x,t) \in M \times (\R/\Z)$ of $\tilde M$, where $X,Y \in T_x M$ are tangent vectors to $M$ at $x$, and $u,v \in T_t \R/\Z \equiv \R$ are tangent vectors to $\R/\Z$ at $t$.  Clearly, $(\tilde M, \tilde g)$ is an $m+1$-dimensional Riemannian manifold with a projection map $\Pi \colon \tilde M \to M$ to $M$ defined by $\Pi(x,t) \coloneqq x$.  If we then define the pullbacks
\begin{align*}
\tilde U(y) &\coloneqq \Pi^* U(y) \\
\tilde V(y) &\coloneqq \Pi^* V(y) \\
\tilde P'(y,y') &\coloneqq \Pi^* P'(y,y')
\end{align*}
for $y, y' \in \R^n$, then $\tilde U \colon \R^n \to \Gamma(T\tilde M)$ is linear, $\tilde V \colon \R^n \to \Gamma(T^*\tilde M)$ is injective and linear, and $\tilde P' \colon \R^n  \times \R^n \to C^\infty(M)$ is symmetric and bilinear.
Similarly, if we define the $m+1$-dimensional volume form
$$ \widetilde{\mathrm{vol}}^{m+1} \coloneqq \Pi^* \mathrm{vol}^m \wedge dt $$
where $dt$ is the derivative of the second local coordinate $t \colon (x,t) \mapsto t$ on $M \times \R/\Z$, then $\widetilde{\mathrm{vol}}^{m+1}$ is a volume form.  From pulling back \eqref{g1}-\eqref{g4} (and working in coordinates if desired), we see that
\begin{align*}
\tilde V(B(y,y)) + \tilde U(y) \invneg d\tilde V(y) &= -d\tilde P'(y,y)\\ 
\tilde V(y) &= \tilde g \cdot \tilde U(y) \\
\mathrm{div}_{\widetilde{\mathrm{vol}}^{m+1}} \tilde u_a &= 0
\end{align*}
on $M$ for all $y \in \R^n$.
On the other hand, a direct computation in coordinates using \eqref{gxy} reveals that
$$ \mathrm{det}_{\widetilde{\mathrm{vol}}^{m+1}} \tilde g = (\mathrm{det}_{\mathrm{vol}^m} g) (\mathrm{det}_{\mathrm{vol}^m} g)^{-1} = 1$$
and hence $\widetilde{\mathrm{vol}}^{m+1}$ is the volume form associated to the Riemannian metric $\tilde g$.  In particular we have
$$ \mathrm{div}_{\tilde g} \tilde u_a = 0$$
and Theorem \ref{first} follows.

It remains to establish Theorem \ref{second}.  This will be the objective of the next three sections of the paper.

\section{Third reduction: eliminating the metric}

In Theorem \ref{second}, the metric $g$ now only appears in a single equation \eqref{g2}.  This equation forces the ``Gram bilinear form'' 
$$(y,y') \mapsto V(y)(U(y'))$$
from $\R^n \times \R^n \to \R$ to be symmetric and positive semi-definite, since 
$$ V(y)(U(y')) = g( U(y), U(y')) = g(U(y'), U(y)) = V(y')(U(y)) $$
and thus
$$ V(y)(U(y)) = g(U(y), U(y)) \geq 0.$$
We can reverse this implication if we assume that this matrix is in fact everywhere positive definite (not just positive semi-definite), allowing us to eliminate the role of the metric $g$.  More precisely, we deduce Theorem \ref{second} from

\begin{theorem}[Third reduction]\label{third}  Let $B \colon \R^n \times \R^n \to \R^n$ be a symmetric bilinear map obeying \eqref{yo}.  Then there exists a compact\footnote{As in the introduction, we use ``compact manifold'' as an abbreviation for ``compact connected smooth oriented manifold without boundary''.} manifold $M$ of some dimension $m$, an injective linear map $V \colon \R^n \to \Gamma(T^* M)$, a linear map $U \colon \R^n \to \Gamma(TM)$, a symmetric bilinear map $P' \colon \R^n \times \R^n \to C^\infty(M)$, and an everywhere positive volume form $\mathrm{vol}^m \in \Gamma(\bigwedge^m T^* M)$ obeying the equations
\begin{align}
V(B(y,y)) + U(y) \invneg dV(y) &= -dP'(y,y) \label{h1a}\\
\mathrm{div}_{\mathrm{vol}^m} U(y) &= 0 \label{h4a}
\end{align}
on $M$ for all $y \in M$, and such that the Gram bilinear form
\begin{equation}\label{gram}
(y,y') \mapsto V(y)(x)(U(y')(x))
\end{equation}
is symmetric and positive definite for every $x \in M$.
\end{theorem}

Let us now see how Theorem \ref{third} implies Theorem \ref{second}.  Let $n, B$ obey the hypotheses of Theorem \ref{second}, and let $M, m, V, U, P', \mathrm{vol}^m$ be the objects associated to $n, B$ by Theorem \ref{third}.  From the positive definiteness of \eqref{gram}, we have that
$$ V(y)(x)(U(y)(x)) > 0$$
for all $y \in \R^n$ and $x \in M$.  This implies that the maps $U \colon \R^n \to \Gamma(T M)$ and $V \colon \R^n \to \Gamma(T^* M)$ are injective, in fact their pointwise evaluations $y \mapsto U(y)(x)$ and $v \mapsto V(y)(x)$ are injective for each $x \in M$.  (In particular, this forces $n \leq m$.)

Let $h$ be an arbitrary Riemannian metric on $M$.  At each point $x$ of $M$, let $U_x \subset T_x M$ be the $n$-dimensional linear space
$$ U_x \coloneqq \{ U(y)(x): y \in \R^n\},$$
and let $U_x^\perp \subset T_x M$ be the orthogonal complement of $U_x$ in $T_x M$ with respect to the metric $h$.  Then $UM \coloneqq (U_x)_{x \in M}$ and $UM^\perp \coloneqq (U_x^\perp)_{x \in M}$ are smooth subbundles of $TM$, whose direct sum is $T M$.  Let $C>0$ be a large constant, and define the symmetric $(0,2)$-tensor $g$ on $M$ by the formula
\begin{align*}
g( U(y) + X, U(y') + Y) &\coloneqq V(y)(U(y')) + V(y)(Y) + V(y')(X) + C h(X,Y) 
\end{align*}
whenever $y,y' \in \R^n$ and $X, Y \in \Gamma(UM^\perp)$ are vector fields in $UM^\perp$.  This clearly defines a symmetric $(0,2)$ tensor, with the property that $g(U(y),X) = V(y)(X)$ for all $y \in \R^n$ and $X \in \Gamma(TM)$.  Now we claim that $g$ is positive definite (and hence a Riemannian metric).  Indeed, for any $y \in \R^n$ and $X \in \Gamma(UM^\perp)$, we have
$$
g( U(y) + X, U(y) + X) = V(y)(U(y)) + 2 V(y)(X) + C h(X,X).
$$
Since \eqref{gram} is positive definite, and $M$ is compact, we have
$$ V(y)(U(y)) \geq \delta \|y\|^2$$
for some constant $\delta>0$.  The claimed positive definiteness now follows from the Cauchy-Schwarz inequality (and the fact that the $h$ is invertible), if $C$ is chosen large enough.  From construction we have \eqref{g2} for all $y \in \R^n$, and the claim follows.

It remains to establish Theorem \ref{third}.  This will be the objective of the next two sections of the paper.  

\section{Fourth reduction: a simplifying ansatz}

We now give an ansatz for the unknown fields $V, P'$ in terms of a bilinear map $F \colon \R^n \times \R^R \to C^\infty(M)$ that simplifies the equations significantly.  More precisely, we deduce Theorem \ref{third} from

\begin{theorem}[Fourth reduction]\label{fourth}  Let $B \colon \R^n \times \R^n \to \R^n$ be a symmetric bilinear map obeying \eqref{yo}.  Then there exists a natural number $r$, a compact manifold $M$ of some dimension $m$, a linear map $U \colon \R^n \to \Gamma(TM)$, a bilinear map $F \colon \R^n \times \R^r \to C^\infty(M)$, and an everywhere positive volume form $\mathrm{vol}^m \in \Gamma(\bigwedge^m T^* M)$ obeying the equations
\begin{align}
F(B(y,y),z) + dF(y,z)(U(y)) &= 0 \label{h1} \\
\mathrm{div}_{\mathrm{vol}^m} U(y) &= 0, \label{h4}
\end{align}
on $M$ for all $y \in \R^n$ and $z \in \R^r$, and such that for every non-zero $y \in \R^n$ and any $x \in M$, the map $z \mapsto F(y,z)(x)$ is injective.
\end{theorem}

Let us now see how Theorem \ref{fourth} implies Theorem \ref{third}.  Let $n, B$ obey the hypotheses of Theorem \ref{third}, and let $M, m, r, U, F, \mathrm{vol}^m$ be the objects associated to $n, B$ by Theorem \ref{fourth}.  

We introduce the compact $m+r$-dimensional manifold
$$ \tilde M \coloneqq M \times (\R/\Z)^r$$
with coordinates $(x, t_1,\dots,t_r)$ with $x \in M$ and $t_1,\dots,t_r \in \R/\Z$.  There is an obvious projection map $\Pi \colon \tilde M \to M$, defined by
$$ \Pi(x, t_1,\dots,t_R) \coloneqq x.$$
We now define the linear map $\tilde U \colon \R^n \to \Gamma(T\tilde M)$, the linear map $\tilde V \colon \R^n \to \Gamma(T^* \tilde M)$, the symmetric bilinear map $\tilde P \colon \R^n \times \R^n \to \in C^\infty(\tilde M)$, and the volume form $\widetilde{\mathrm{vol}}^{m+r} \in \Gamma(\bigwedge^{m+r} T^* \tilde M)$ by the formulae
\begin{align}
\tilde U(y) &\coloneqq \Pi^* U(y) + \sum_{i=1}^r (\Pi^* F(y,e_i)) \frac{d}{dt_i} \nonumber\\
\tilde V(y) &\coloneqq \sum_{i=1}^r (\Pi^* F(y,e_i)) dt_i \nonumber\\
\tilde P(y,y') &\coloneqq \frac{1}{2} \sum_{i=1}^r \Pi^*(F(y,e_i) F(y',e_i)) \label{p-def}\\
\widetilde{\mathrm{vol}}^{m+R} &\coloneqq \Pi^* \mathrm{vol}^m \wedge dt_1 \wedge \dots \wedge dt_R\nonumber
\end{align}
for $y,y' \in \R^n$, where $e_1,\dots,e_r$ is the standard basis for $\R^r$.
Direct calculation using \eqref{h4} yields the equations
$$
\tilde V(y)(\tilde U(y')) = \Pi^* \left(\sum_{i=1}^r F(y,e_i) F(y',e_i)\right)
$$ 
and
$$
\mathrm{div}_{\widetilde{\mathrm{vol}}^{m+R}} \tilde U(y) = 0$$
for $y, y' \in \R^n$; in particular, the Gram bilinear form $(y,y') \mapsto \tilde V(y)(x)(\tilde U(y')(x))$ is symmetric and strictly positive definite for every $x \in M$, since by hypothesis one cannot have $F(y,e_i)(x)$ vanish for all $i=1,\dots,R$ if $y$ is non-zero.  We may also compute the exterior derivative of $\tilde V(y)$ as
$$
d \tilde V(y) = \sum_{i=1}^r \Pi^*(dF(y,e_i)) \wedge dt_i  
$$
and hence
$$
\tilde U(y) \invneg d \tilde V(y) = \sum_{i=1}^R \Pi^*\left(dF(y,e_i)(U(y))\right) dt_r - \Pi^*( F(y,e_i) dF(y,e_i) );
$$
using \eqref{h1}, \eqref{p-def} we conclude that
$$
\tilde V(B(y,y)) + \tilde U(y) \invneg d\tilde V(y) = 
- d \tilde P(y,y)
$$
and Theorem \ref{third} follows.  

It remains to establish Theorem \ref{fourth}.  This will be the objective of the next section of the paper.

\section{An exact solution}

The system \eqref{h1}, \eqref{h4} appears to be rather overdetermined when $n$ is large; in coordinates, one is asking to solve on the order of $n^2 r$ equations, but one only has about $nr$ independent scalar functions.  Remarkably, though, there is still a non-trivial solution to this system that can be described exactly; this solution evades the overdeterminacy by being highly symmetric.  

We first make a simple observation.  Let $\mathfrak{so}(n)$ denote the space of skew-adjoint maps $q \colon \R^n \to \R^n$; this is of course the Lie algebra of the compact Lie group $SO(n)$ of special orthogonal transformations $Q \colon \R^n \to \R^n$, which is connected and orientable and thus considered a compact manifold in our notation.  We can relate this Lie algebra $\mathfrak{so}(n)$ to the cancellation condition \eqref{yo}:

\begin{lemma}\label{stax}  Let $B \colon \R^n \times \R^n \to \R^n$ be a symmetric bilinear form.  Then $B$ obeys \eqref{yo} if and only if there exists a linear map $S \colon \R^n \to \mathfrak{so}(n)$ such that
\begin{equation}\label{byya}
 B(y,y) = S(y)(y)
\end{equation}
for all $y \in \R^n$.
\end{lemma}

\begin{proof} Clearly, if \eqref{byya} holds then \eqref{yo} holds, thanks to the skew-adjointness of $S(y)$.  Conversely, suppose that $B$ obeys \eqref{yo}.  We define $S \colon \R^n \to \mathfrak{so}(n)$ via duality, setting
\begin{equation}\label{sdef}
\langle S(y_1) y_2, y_3 \rangle_{\R^n} \coloneqq \frac{2}{3} \left( \langle B(y_1, y_2), y_3 \rangle_{\R^n} - \langle B(y_1, y_3), y_2 \rangle_{\R^n} \right)
\end{equation}
for $y_1,y_2,y_3 \in \R^n$.  Clearly $S(y_1)$ is skew-adjoint for any $y_1 \in \R^n$.  For any $y,z \in \R^n$, by applying \eqref{yo} with $y$ replaced by $y+tz$ for $t \in \R$ and extracting the coefficient linear in $t$, we see that
$$\langle B(y,y), z \rangle_{\R^n} + 2 \langle B(y,z), y \rangle_{\R^n} = 0$$
and hence from setting $y_1=y_2=y$ and $y_3=z$ in \eqref{sdef}, we conclude after some algebra that
$$ \langle S(y)(y), z \rangle_{\R^n} = \langle B(y,y), z \rangle_{\R^n}$$
for all $y,z \in \R^n$, and \eqref{byya} follows.
\end{proof}

Now we can prove Theorem \ref{fourth}.  Let $n, B$ be as in that theorem, and let $S \colon \R^n \to \mathfrak{o}(n)$ be the map provided by Lemma \ref{stax}.  We set $M$ to be the special orthogonal group $M \coloneqq SO(n)$ (hence $m \coloneqq \frac{n(n-1)}{2}$), and $\mathrm{vol}^m$ to be a Haar measure on $M$ (it will be irrelevant how one normalises this measure, but one can for instance take the probability Haar measure).  For each $y \in \R^n$, we set $U(y)$ to be the right-invariant vector field on $M$ whose value at any orthogonal transformation $Q \in M$ is given by
$$ U(y) Q = S(y) Q$$
(here we view $M$ as embedded in the vector space $\mathrm{End}(\R^n)$ of $n \times n$ matrices, and the tangent space $T_Q M$ of $M$ at $Q$ as a subspace of that vector space).  Since $S(y)$ lies in the Lie algebra of $O(n)$, the flow along $S(y)$ preserves Haar measure, and hence the Lie derivative of $\mathrm{vol}^m$ along $U(y)$ vanishes; in other words, we have \eqref{h4}.

We set $r \coloneqq n$, and set $F \colon \R^n \times \R^n \to C^\infty(M)$ to be the map
$$ F(y,z)(Q) \coloneqq \langle y, Q z \rangle_{\R^n}.$$
Clearly, if $y$ is non-zero and $Q$ is orthogonal, then $F(y,z)(Q)$ cannot vanish for all $z \in \R^r$.  We have
\begin{align*}
dF(y,z)(U(y)) (Q) &= \langle y, S(y) Q z \rangle_{\R^n} \\
&= -\langle S(y) y, Qz \rangle_{\R^n} \\
&= -\langle B(y,y), Qz \rangle_{\R^n}
\end{align*}
by skew-adjointness of $S(y)$ and \eqref{byya}, and \eqref{h1} follows. This proves Theorem \ref{fourth}.

\begin{remark}  This observation was communicated to the author by Tobias Diez.  One can specialise Theorem \ref{fourth} to the case when the bilinear form $B$ arises from the Euler equation on a compact Lie group $G$ whose associated Lie algebra ${\mathfrak g}$ (which we identify with $\R^n$) is equipped with a scalar product $\langle, \rangle_{\mathfrak g}$.  In this case one has
$$ B(y,y) = \mathrm{ad}^*_y y$$
where $\mathrm{ad}^*_y: \R^n \to \R^n$ is the dual of the adjoint action $\mathrm{ad}_y: \R^n \to \R^n$ with respect to the scalar product; since $\mathrm{ad}_y y = 0$, one has the cancellation condition \eqref{yo} with the indicated scalar product.  In this case, one can modify the above proof of Theorem \ref{fourth} by setting $M \coloneqq G$ with Haar measure, $\R^r \coloneqq \R^n = {\mathfrak g}$, $F(y,z)(g) \coloneqq \langle y, \mathrm{Ad}_g z \rangle_{\mathfrak g}$, and $U(y)$ to be the right-invariant vector field on $G$ associated to $y$, thus $U(y)(g) \coloneqq y g$.  A brief computation analogous to the one above then shows that the conclusions of Theorem \ref{fourth} are obeyed.
\end{remark}

% You may incorporate your references as follows in your main tex file.
% Using BibTex is not recommended but can be handled.

\medskip
% The data information below will be filled by AIMS editorial staff
Received xxxx 20xx; revised xxxx 20xx.
\medskip

\end{document}